\documentclass[a4paper,12pt]{article}
\usepackage{bm}
\usepackage{amsmath}
\usepackage{amsthm}
\usepackage{wasysym}
\usepackage{amssymb}
\usepackage{amsfonts}
\usepackage{graphicx}
\usepackage[english]{babel}
\usepackage[utf8]{inputenc}
\usepackage[T1]{fontenc}
\usepackage{mathrsfs}
\usepackage{enumerate}
\usepackage{version}
\usepackage{calc}
\usepackage{subfigure}

 \usepackage{pstricks,pst-math,pst-xkey}
 \usepackage{float}
 \usepackage{indentfirst}
\newtheorem{thm}{Theorem}

\newtheorem{lem}{Lemma}

\newtheorem{definition}{Definition}

\theoremstyle{remark}
\newtheorem{remark}{Remark}

\def\a{\alpha}

\def\ve{\varepsilon}

\def\d{\partial}

\def\t{\tau}

\def\F{\mathcal{F}}

\def\f{\varphi}
\def\Z{\mathbb{Z}}

\def\R{\mathbb{R}}
\def\E{\mathbb{E}}

\def\P{\mathbb{P}}

\author{Oleksandr~\textsc{Misiats}\\
 Department of Mathematics, Purdue University\\
West Lafayette, IN, 47907, USA
\\{\tt omisiats@purdue.edu}
\and Oleksandr~\textsc{Stanzhytskyi}\\
Department of Mathematics, Kiev National University, Kiev, Ukraine
\\{\tt ostanzh@gmail.com}
\and Nung Kwan~\textsc{Yip}\\
 Department of Mathematics, Purdue University\\
West Lafayette, IN, 47907, USA
\\{\tt yip@math.purdue.edu}
}

\title{Existence and uniqueness of invariant measures for stochastic reaction-diffusion equations in unbounded domains}

\begin{document}
\maketitle

\begin{abstract}
In this paper we investigate
 the long-time behavior of stochastic reaction-diffusion equations of the type
$du = (Au + f(u))dt + \sigma(u) dW(t)$, where $A$ is an elliptic operator, $f$ and $\sigma$ are nonlinear maps and $W$ is an infinite dimensional nuclear Wiener process. The emphasis is on unbounded domains.
Under the assumption that the nonlinear function $f$ possesses certain
dissipative properties, this equation is known to have a solution with
an expectation value which is uniformly bounded in time.
Together with some compactness property, the existence of such a solution implies the existence of an invariant measure which is an important step in establishing the ergodic behavior of the underlying physical system. In this paper we expand the existing classes of nonlinear functions $f$ and $\sigma$ and elliptic operators $A$ for which the invariant measure exists, in particular, in unbounded domains.
We also show the uniqueness of the invariant measure for an equation defined on the upper half space if $A$ is the
Shr\"{o}dinger-type operator $A = \frac{1}{\rho}(\text{div} \rho \nabla u)$ where $\rho  = e^{-|x|^2}$ is the Gaussian weight.

\end{abstract}

\section{Introduction and Main Results}
We study the long time behavior of the equation
\begin{equation}\label{eqn1}
\begin{cases}
\frac{\d}{\d t} u(t,x) = A u(t,x) + f(x,u(t,x)) + \sigma(x,u(t,x)) \dot{W}(t,x), t > 0, x \in G;\\
u(0,x) = u_0(x).
\end{cases}
\end{equation}
Here $G \subset \R^d$ is a (possibly unbounded) domain, $A$ is an elliptic operator, $f$ and $\sigma$ are measurable real functions, and the Gaussian noise $\dot{W}(t,x)$ is white in time and colored in space. In particular, we are interested in the existence and uniqueness of invariant measures for the equation (\ref{eqn1}).

Equations of this type model the behavior of various dynamical systems in physics and mathematical biology. For instance, this equation describes the
well known Hodgkin-Huxley model in neurophysiology (where $u$ is the electric potential on nerve cells \cite{HodHux}), as well as the Dawson and Fleming model of population genetics \cite{DawFle} ($u(t, \cdot)$ is the mass distribution of population). Besides, the equation (\ref{eqn1}) with infinite  dimensional noise is an interesting object from the mathematical point of view since its analysis involves subtle interplay between PDE and probabilistic techniques.

Reaction-diffusion equations of type (\ref{eqn1}) have been extensively studied by a variety of authors. The analysis of the long-time behavior of solutions of (\ref{eqn1}) is a nontrivial question even in the deterministic case $\sigma(x,u) \equiv 0$. This question was addressed, for example by
N. Dirr and N. Yip \cite{DirYip} and references therein. In their work, the authors describe a certain class of nonlinearities $f(x,u)$, for which the deterministic equation (\ref{eqn1}) admits a bounded solution (as $t \to \infty$), while for a different class of nonlinearities all solutions of the deterministic equation (\ref{eqn1}) have linear growth in $t$ (and hence are not uniformly bounded). The transition between those two classes of nonlinearities is also studied in the paper.

A comprehensive study of stochastic equation (\ref{eqn1}) has been performed by G. Da Prato and J. Zabczyk (\cite{DapZab92, DapZab96} and references therein). The ergodic properties of the solutions of (\ref{eqn1}) is a question of separate interest in these works. This question was addressed from the point of view of the existence of an invariant measure for (\ref{eqn1}), which is a key step in the study of the ergodic behavior of the underlying physical systems \cite[Theorems 3.2.4, 3.2.6]{DapZab96}.
Based on the pioneering work of Krylov and Bogoliubov \cite{KryBog},
the authors suggested the following approach to establish the existence of invariant measures:
\begin{itemize}
\item Establish the compactness and Feller property of the semigroup
generated by $A$;

\item Establish the existence of a solution which is bounded for  $t \in [0, \infty)$ in certain probability sense.
\end{itemize}
The existence of invariant measures using the aforementioned procedure
was established in \cite{MasSie99, GolMas01, BrzGat99}, in particular, in the case when $A = \Delta$ and $G$ is a bounded domain.

A different approach to the existence of invariant measures, based on the coupling method, was used by Bogachev and Roechner \cite{BogRoe01} and
C. Mueller \cite{Mue93}. This method can be applied even for space-white noise but only in the case when the space dimension $d$ is one.

The existence and uniqueness of the solutions of stochastic
reaction-diffusion equations in bounded domains with Dirichlet boundary
condition, as well as the existence of an invariant measure was studied by S. Cerrai in \cite{Cer03, Cer01, CerExtr} and references therein.

The question of the existence of invariant measures in {\em unbounded domains} with $A = \Delta$ was studied in \cite{DapZab96, Hairer, TesZab98, AssMan01}.
The key condition for the existence of a solution bounded in probability,
and hence the existence of an invariant measure in these works is the
following dissipation condition for the nonlinearity $f$: for some $k>0$,
\begin{equation}\label{diss}
\begin{cases}
f(u) \geq -k u - c, u \leq 0;\\
f(u) \leq -k u + c, u \geq 0.
\end{cases}
\end{equation}

To the best of our knowledge, the only case the existence of an invariant
measure in $\R^d$ is proved when $f(u)$ does not satisfy the dissipativity condition (\ref{diss})
is the work of Assing and Manthey \cite{AssMan01}.
For spatial dimensions three or higher, these authors show the existence
of an invariant measure for (\ref{eqn1}) if $f(u) \equiv 0$
and $\sigma(u)$ is a  Lipschitz function of $u$ with a sufficiently small Lipschitz constant.
One of the goals of the present work is to extend the results of
\cite{AssMan01} to incorporate $f$ which might not satisfy the condition
\eqref{diss}.

We establish two types of existence results for invariant measures
in unbounded domains. The first is to make use
of the boundedness and compactness property of the solution.
The dissipativity required comes not
from the nonlinear function $f$ but from the decaying property of the Green's function in three and higher dimensions in $\R^d$. The second is to
make use of the exponential stability of the equation.
This approach also gives the uniqueness of the invariant measure.
Both strategies are similar to
\cite{daPratoGatarekZabczyk, daPratoZabczykBdErgodic}
while the analytical framework is different.

Before describing our results, we introduce some weighted $L^2$-space.
Let $\rho$ be a non-negative continuous $L^1(\R^d) \bigcap L^\infty(\R^d)$ function.
Following \cite{TesZab98}, we call $\rho$ to be an {\it admissible} weight if for every $T>0$ there exists $C(T)>0$ such that
\[
G(t, \cdot) * \rho \leq C(T) \rho, \ \forall t \in [0,T], \,\text{where}\,\,G(t,x) = \frac{1}{(4 \pi t)^{d/2}}e^{ - \frac{|x|^2}{4 t}}.
\]
Some examples of admissible weights include $\rho(x) = exp(-\gamma |x|)$ for $\gamma>0$, and $\rho(x) = (1+|x|^n)^{-1}$ for $n>d$.

For an admissible weight $\rho$, define
\begin{equation}\label{weight_L2def}
H = L^2_{\rho}(\R^d):=\{w: \R^d \to \R, \int_{\R^d} |w(x)|^2 \rho(x)\, dx < \infty\}
\end{equation}
and
\[
\|w\|^2_{H}:= \int_{\R^d} |w(x)|^2 \rho(x)\, dx.
\]
The choice of $\rho$ is more flexible for the first part while it is
quite specific for the second. The noise process $W$ is defined and
constructed at the beginning of
Sections \ref{SecInvMeasEntire} and \ref{SecInvMeasUpperHalf}.

Our first set of results is stated as follows.
\begin{thm}\label{thm1}
Let $d\geq 3$. Assume
\begin{enumerate}
\item $\sigma:\R^d \times \R \to \R$ satisfies $|\sigma(x,u_1) - \sigma(x,u_2)| \leq c|u_1 - u_2|$ and $|\sigma(x,u)| \leq \sigma_0$ for some $\sigma_0 > 0$.
\item $f:\R^d \times \R \to \R$ satisfies $|f(x,u_1) - f(x,u_2)| \leq c|u_1 - u_2|$ and there exists $\f(x) \in L^1(\R^d) \cap L^{\infty}(\R^d)$ such that
\begin{equation}\label{bound}
|f(x,u)|\leq \f(x), \forall (x,u) \in \R^d \times \R.
\end{equation}
\end{enumerate}
Let $u(t,x)$ be a solution of (\ref{eqn1}) with $\E\|u(0,x)\|^2_{L^2(\R^d)}<\infty$,
then we have
\[
\sup_{t \geq 0} \E\|u(t,x)\|^2_{H} < \infty.
\]
\end{thm}

To state our second result, for simplicity, we write $f(x,u(x))$ as $f$
which maps $L^2$ to $L^2$.
\begin{thm}\label{thm2}
Let $d\geq 3$. Assume
\begin{itemize}
\item[(i)] $\forall u,v \in L^2(\R^d)$, $\|f(u)-f(v)\|_{L^2(\R^d)} \leq C \|u-v\|_{L^2(\R^d)}$ and $\|\sigma(\cdot,u(\cdot)) - \sigma(\cdot,v(\cdot))\|_{L^2(\R^d)} \leq C \|u-v\|_{L^2(\R^d)}$;

\item[(ii)]
For some $N > 0$, $f(u) = 0$ if $\|u\|_{L^2(\R^d)} \geq N$.

\item[(iii)] There exists $\psi(x) \in L^2(\R^d)$ such that
\begin{equation}\label{boundsigma}
|\sigma(x,u)|\leq \psi(x), \forall (x,u) \in \R^d \times \R.
\end{equation}
\end{itemize}
Let $u(t,x)$ be a solution of (\ref{eqn1}) with $u(0,x) = u_0(x)\in L^2(\R^d)$. Then
\begin{equation}\label{mainRth2}
\sup_{t \geq 0} \E\|u(t,x)\|^2_{L^2(\R^d)} < \infty.
\end{equation}
\end{thm}
\begin{remark}
Note that (\ref{mainRth2}) implies $\sup_{t \geq 0} \E\|u(t,x)\|^2_{L^2_{\rho} (\R^d)} < \infty$ for any weight $\rho\in L^\infty(\R^d)$.
\end{remark}

\begin{remark}
For both of the above theorems, the Lipschitz conditions for $f$ and $\sigma$
are mainly used for the existence and uniqueness of the solutions while their
global bounds and contraints are for proving the uniform boundedness in time.
\end{remark}

\begin{remark}
Comparing with the results of \cite{AssMan01},
we do not require the smallness of the Lipschitz constants of $f$ and
$\sigma$. These are replaced by their somewhat more global conditions.
\end{remark}

Roughly speaking, in the case $d \geq 3$, the Laplace operator has sufficiently strong dissipative properties which compensate for the lack of dissipation coming from $f(u)$. These results, in conjunction with the compactness property of the semigroup for the Laplace operator in some weighted space defined on $\R^d$, yield the existence of an invariant measure for (\ref{eqn1}) using the Krylov-Bogoliubov
approach \cite[Theorem 6.1.2]{DapZab96}.

In the analysis of the ergodic behavior of dynamical systems, the
uniqueness of invariant measures is a key step.
As shown in \cite[Theorem 3.2.6]{DapZab96}, the uniqueness of the invariant measure implies that the solution process is ergodic. However, establishing the
uniqueness property of the invariant measure is highly nontrivial.
One approach, illustrated in \cite[Chapter 7]{DapZab96}, shows that
the uniqueness is a consequence of a strong Feller property and
irreducibility. Typically, in order to apply this result, one needs to impose rather restrictive conditions both on the diffusion coefficient and on the semigroup $\{S(t)\}_{t\geq 0}$ generated by the elliptic
operator. In particular, the diffusion operator has to be bounded and
non-degenerate, while the semigroup
has to be square integrable in some Hilbert-Schmidt norm
\cite[Hypothesis 7.1(iv)]{DapZab96}. However, this condition does not hold
for the Laplace operator in unbounded domains.

In the second part of our work, we use a different approach to establish the uniqueness of invariant measures which does not require
\cite[Hypothesis 7.1(iv)]{DapZab96}.
This approach, reminiscent of \cite[Theorem 6.3.2]{DapZab96}, is based on
the fact that if the semigroup has an exponential contraction property
\begin{equation}\label{contraction}
\|S(t) u\| \leq M e^{-\gamma t} \|u\|,
\end{equation}
for some $M,\,\,\gamma>0$, then the corresponding dynamical system
possesses a unique solution which is stable and uniformly bounded in
expectation.
This solution is utilized in the proof of the uniqueness of the invariant
measure.
The condition (\ref{contraction}) holds in particular if $A$ is the
Laplace operator $\Delta$ in a bounded domain $G$ with Dirichlet boundary
condition. Our result however, deals with an example when $G$ is
unbounded.

Consider
\begin{equation}\label{eqn1semi}
\begin{cases}
\frac{\d}{\d t} u(t,x) = A u(t,x) + f(x,u(t,x)) + \sigma(x,u(t,x)) \dot{W}(t,x), t > 0,\,\,\,x \in G;\\
u(0,x) = u_0(x)
\end{cases}
\end{equation}
\label{eqn1semi.notation}
where
\begin{itemize}
\item $G = \R^d_+:=\{x = (x_1, x_2,..., x_d) \in \R^d, x_d>0\}$;
\item $\rho(x):= e^{-|x|^2}$, $x \in G$;
\item $H = L^2_{\rho} (G)$;
\item
\begin{equation}\label{Srodinger}
A u := \frac{1}{\rho} \text{div}(\rho \nabla u);
\end{equation}
\item $D(A) := H^2_{\rho}(G) \cap H^1_{0,\rho}(G)$;
\item $f(x,u): G \times \R \to \R$ and $\sigma(x,u): G \times \R \to \R$ satisfy
\begin{equation}\label{Lip}
|f(x,u_1)-f(x,u_2)| \leq L |u_1-u_2|; \ \ |\sigma(x,u_1)-\sigma(x,u_2)| \leq L |u_1-u_2|,
\end{equation}
with Lipschitz constant $L$ independent of $x$;
\item
\begin{equation}\label{linf}
f(x,0) \in L^{\infty}(G) \text{ and } \sigma(x,0)\in L^{\infty}(G).
\end{equation}
\end{itemize}
Note that the elliptic operator $A$ given by (\ref{Srodinger})
appears in quantum mechanics in the analysis of the energy levels of
harmonic oscillator.

Under the assumptions above, the initial value problem (\ref{eqn1semi}) is well-posed (see Theorem \ref{thm3exist}, p. \pageref{thm3exist}).
Our main result for \eqref{eqn1semi} is the following theorem.
\begin{thm}\label{thm5nonlin}
Assume the Lipschitz constant $L$ in \eqref{Lip}
is sufficiently small (see (\ref{smallness}) and (\ref{small2}) below). Then the equation (\ref{eqn1semi}) has a unique solution $u^{*}(t,x)$ which is defined for all $t \in \R$ and satisfies
\[
\sup_{t \in \R} \E \|u^{*}(t,x)\|_H^2 < \infty.
\]
This solution is exponentially stable (in the sense of Definition
\ref{defExpStab}, page \pageref{defExpStab}).
\end{thm}

In Section \ref{SecUniqInvMeas}, the above solution will be used to
prove the existence and uniqueness of the invariant measure for
\eqref{eqn1semi}. In fact, it will be shown that $u^*$ is a stationary
random process.

\begin{remark}
Our approach was motivated by the following simple observation: if $v(t,x), x \in [0,1], t \in \R$ solves
\begin{equation}\label{1d}
\begin{cases}
v_t(t,x) = v_{xx}(t,x);\\
v(t,0) = v(t,1) = 0, \ t \in \R;\\
v(0,x) = \f(x), \ x \in [0,1],
\end{cases}
\end{equation}
then the only exponentially stable solution that satisfies
\[
\sup_{t \in \R} \|v(t,x)\|^2_{L^2([0,1])} < \infty
\]
is $v\equiv 0$ (with $\varphi\equiv 0$).
Theorem \ref{thm5nonlin} is an analog of this fact for the
nonlinear stochastic reaction-diffusion equation (\ref{eqn1semi}).
\end{remark}
\begin{remark}
In contrast with Theorems \ref{thm1} and \ref{thm2}, where the condition $d \geq 3$ is essential, here there is no restriction on the spatial dimension.
\end{remark}

The paper is organized as follows. Section 2 deals with the existence of invariant measure for the reaction-diffusion equation \eqref{eqn1}
with $A = \Delta$ in $\R^d$ and $d \geq 3$ (Theorems  \ref{thm1} and \ref{thm2}). Section 3 is devoted to the proof of Theorem \ref{thm5nonlin} for
equation \eqref{eqn1semi}.
The uniqueness of the invariant measure as a consequence of Theorem \ref{thm5nonlin} is established in Section 4.

\section{Invariant measure in the entire space}\label{SecInvMeasEntire}

In this section, we study the problem (\ref{eqn1}) with $A = \Delta$ and $G = \R^d$. Let $\{e_k, k \geq 1\}$ be an orthonormal basis in $L^2(\R^d)$ such that
\begin{equation}\label{basis}
\sup_{k} \|e_k(x)\|_{L^{\infty}(\R^d)} \leq 1.
\end{equation}
We note that such a basis exists. For example, consider
 \[
 e^{(k)}_{n}(x):= \frac{1}{\pi} \{\sin\left(n x\right), \cos\left(n x \right)\}  \chi_{[2 \pi k, 2 \pi(k+1)]}(x), \, n \geq 0, \, k \in \Z,
 \]
 where $\chi_{[2 \pi k, 2 \pi(k+1)]}(x)$ is the characteristic function of $[2 \pi k, 2 \pi (k+1)]$. Clearly, $$\sup_{n \geq 0, k \in \Z} \|e_n^{(k)}(x)\|_{L^{\infty}(\R)} \leq 1,$$ and
 \[
 \bigcup_{n \geq 0, k \in \Z} e_n^{(k)}(x) \text{ is a basis in } L^2(\R).
 \]
 The basis in  $\R^d$ for $d>1$ can be constructed analogously.

We now define the Wiener process $W(t,x)$ as
\begin{equation}\label{Wiener}
W(t,x) := \sum_{k=1}^{\infty}\sqrt{a_k} \beta_k(t) e_k(x)
\end{equation}
with
\[
a:= \sum_{k=1}^{\infty} a_k < \infty
\]
In the above, the $\beta_k(t)$'s are independent standard one dimensional
Wiener processes on $t \geq 0$.
Let $(\Omega, \mathcal{F}, P)$ be a probability space, and
$\mathcal{F}_{t}$ is a right-continuous filtration such that $W(t,x)$ is adapted to $\mathcal{F}_t$ and $W(t)-W(s)$ is independent of $\mathcal{F}_s$ for all  $s<t$.
As shown in \cite[p. 88-89]{DapZab92}, (\ref{Wiener}) is convergent both
in mean square and with probability one.


We next proceed with a rigorous definition of a {\it mild solution} of (\ref{eqn1}) \cite{DapZab92, DapZab96}:

\begin{definition}\label{defMild0}
Let $H$ be a Hilbert space of functions defined on $\R^d$. An $\mathcal{F}_t$-adapted random process $u(t,\cdot) \in H$ is called a mild solution of (\ref{eqn1}) if it satisfies the following integral relation for $t \geq 0$:
\begin{equation}\label{milds}
u(t,\cdot) = S(t)u_0(\cdot) + \int_{0}^{t} S(t-s)f(\cdot, u(s,\cdot)) ds + \int_{0}^{t} S(t-s) \sigma(u(s,\cdot)) d W(s, \cdot)
\end{equation}
where $\{S(t), t \geq 0\}$ is the semigroup for the linear heat equation,
i.e.
\[
S(t) u(x) := \int_{\R^d} G(t,x-y) u(y) dy.
\]
\end{definition}
It was shown (see for example in \cite{Man99, Man01, AssMan01}) that if both $f$ and $\sigma$ are Lipschitz in $u$, the initial value problem (\ref{eqn1}) admits a unique mild solution $u(t,x)$ if $H = L^2_{\rho}(\R^d)$.
Moreover, as proved in \cite[Proposition 2.1]{TesZab98}, if two
non-negative admissible weights $\rho(x)$ and $\zeta(x)$ in $\R^d$ satisfy
\begin{equation}\label{weights}
\int_{\R^d} \frac{\zeta(x)}{\rho(x)} \, dx < \infty,
\end{equation}
then
\begin{equation}
S(t): L^2_{\rho}(\R^d) \to L^2_{\zeta}(\R^d) \text{ is a compact map. }
\end{equation}
Based on this result, the theorem of Krylov-Bogoliuibov yields the
existence of invariant measure on $L_\zeta^2(\R^d)$ provided
\begin{equation}\label{boundedness}
\sup_{t \geq 0} \E\|u(t,x)\|^2_{L^2_{\rho}(\R^d)} < \infty.
\end{equation}
(\cite[Theorem 3.1]{TesZab98} and \cite[Theorem 2]{AssMan01}).
The statements of Theorems \ref{thm1} and \ref{thm2} exactly show the
existence of a solution satisfying the above condition.

We now proceed to the proof of Theorem \ref{thm1}.

\begin{proof}
Let $u(t,x)$ be a solution of (\ref{eqn1}).
Applying the elementary inequality
$(a+b+c)^2 \leq 3 (a^2 + b^2 + c^2)$ to \eqref{milds}, we have
\[
\|u(t,x)\|^2_{H} \leq 3 \Big(I_1(t) + I_2(t) + I_3(t)\Big)
\]
where
\[
I_1(t) = \int_{\R^d}|S(t) u(0,x)|^2 \rho dx;
\]
\[
I_2(t) = \int_{\R^d}\left|\int_{0}^{t} S(t-s)f(x, u(s,x)) ds\right|^2 \rho dx;
\]
\[
I_3(t) = \int_{\R^d}\left|\int_{0}^{t} S(t-s) \sigma(u(s,x)) d W(s, x)\right|^2 \rho dx.
\]
We will show that
\[
\sup_{t \geq 0} \E I_{i}(t) < \infty, \ \ \ i = 1,2,3.
\]

For $I_1$, we have, by the $L^2$-contraction property of $S(t)$ that
\[
\sup_{t \geq 0} \E I_{1}(t) \leq \|\rho\|_{\infty} \sup_{t \geq 0} \E \|S(t) u(0,x)\|^2_{L^2(\R^d)} \leq \|\rho\|_{\infty} \E \| u(0,x)\|^2_{L^2(\R^d)} < \infty.
\]


We next estimate $I_2$ in the following manner:
\[
I_2(t) = \int_{\R^d}\left|\int_{0}^{t} \int_{\R^d}G(t-s,x-y)f(y, u(s,y)) dy ds\right|^2 \rho dx
\]
\[
\leq
2\int_{\R^d}\left|\int_{0}^{t-1} \int_{\R^d}G(t-s,x-y)f(y, u(s,y)) dy ds\right|^2 \rho dx
 \]
 \[
 + 2\int_{\R^d}\left|\int_{t-1}^{t} \int_{\R^d}G(t-s,x-y)f(y, u(s,y)) dy ds\right|^2 \rho dx.
\]
First, using (\ref{bound}), we have
\[
\int_{\R^d}\left|\int_{t-1}^{t} \int_{\R^d}G(t-s,x-y)f(y, u(s,y)) dy ds\right|^2 \rho dx \leq \|\f\|^2_{\infty}  \|\rho\|_{L^1(\R^d)}
\]
Second, consider,
\[
\int_{\R^d}\left|\int_{0}^{t-1} \int_{\R^d}G(t-s,x-y)f(y, u(s,y)) dy ds\right|^2 \rho dx
\]
\[
\leq \int_{\R^d}\left|\int_{0}^{t-1} \int_{\R^d}\frac{1}{(4 \pi (t-s))^{d/2}}e^{ - \frac{|x-y|^2}{4 (t-s)}} \f(y) dy ds\right|^2 \rho dx
\]
\[
\leq \| \rho \|_{L^1(\R^d)}  \|\f\|^2_{L^1(\R^d)}  \left|\int_{0}^{t-1}\frac{ds}{(4 \pi (t-s))^{d/2}}\right|^2.
\]
Therefore
\[
\sup_{\t \geq 0} \E I_2(t) \leq \|\f\|^2_{\infty}  \|\rho\|_{L^1(\R^d)}  + \frac{1}{(4 \pi)^{d/2}} \| \rho \|_{L^1(\R^d)}  \|\f\|^2_{L^1(\R^d)} \int_{1}^{\infty}\frac{d \tau}{\tau^{d/2}} <  \infty
\]
where the condition $d\geq 3$ is used in the last step.

It remains to show that
$\displaystyle \sup_{t \geq 0} \E I_3(t) < \infty$.
First note that
\[
\E\left|\int_{0}^{t} \int_{\R^d}G(t-s,x-y) \sigma(y,u(s,y)) dy dW(s,y)\right|^2
\]
\[
= \E  \int_{0}^{t} \sum_{k=1}^{\infty} a_k \left(\int_{\R^d} G(t-s,x-y) \sigma(y,u(s,y)) e_k(y) dy\right)^2 ds
\]
\[
= \E \int_{0}^{t-1} \sum_{k=1}^{\infty} a_k \left(\int_{\R^d} G(t-s,x-y) \sigma(y,u(s,y)) e_k(y) dy\right)^2 ds
\]
\[
+ \E \int_{t-1}^{t} \sum_{k=1}^{\infty} a_k \left(\int_{\R^d} G(t-s,x-y) \sigma(y,u(s,y)) e_k(y) dy\right)^2 ds
\]
\[
\leq \sigma_0^2 \int_{0}^{t-1} \int_{\R^d} G^2(t-s,x-y) dy \sum_{k=1}^{\infty} a_k \int_{\R^d} e_k^2(y) dy
+ \sigma_0^2 \int_{t-1}^{t} \sum_{k=1}^{\infty} a_k \left(\int_{\R^d} G(t-s,x-y) dy\right)^2 ds.
\]
\[
\leq
a \sigma_0^2 \left( \int_{0}^{t-1} \int_{\R^d} G^2(t-s,y) dy ds + 1 \right)
\leq
a \sigma_0^2 \left( \int_{0}^{t-1} \frac{1}{(t-s)^{\frac{d}{2}}} ds + 1 \right)
\leq C < \infty
\]
Therefore,
\[
\E I_3(t)
= \int_{\R^d}
\E\left|\int_{0}^{t} \int_{\R^d}G(t-s,x-y) \sigma(y,u(s,y)) dy dW(s,y)
\right|^2 \rho dx
\leq
C\|\rho\|_{L^1(\R^d)}
\]
which is uniformly bounded independent of $t$, thus concluding the proof.
\end{proof}

We next prove Theorem \ref{thm2}.
\begin{proof}
(For simplicity, we omit the $x$ variable in $f$ and $\sigma$.)
Let $\| u(0,x) \|_{L^2(\R^d)} = Z$ and $M:=\max\{Z,N\}$ where $N$ is given by the condition (ii). For given $t>0$, consider the random variable
\[
\tau =
\begin{cases}
\sup \{ 0 < s \leq t: \|u(s,x)\|_{L^2(\R^d)} = M + 1\} \text{ if the given set is nonempty }\\
t, \text{ otherwise. }
\end{cases}
\]
Introduce
\[
C :=\{\omega \in \Omega: \|u(t,x,\omega)\|_{\R^d} > M+1\}
\]
It follows from the local H\"{o}lder continuity in time of solutions of (\ref{eqn1}) \cite{SanSar02} that $\|u(s,x,\omega)\|_{L^2(\R^d)}$ is continuous in $s$. Therefore, for
$\omega \in \left\{\tau(\omega) < t\right\}\bigcap C$, we have
\[
\|u(s,x,\omega)\|_{L^2(\R^d)}> M+1, \ \ s \in (\tau, t]
\]
Note that a stochastic integral $f(t):=\int_{0}^{t}g(s)dW(s)$ is an a.e. continuous function of $t$. Thus if $\tau$ is another random variable,
the expression $f(\tau)$ is well defined \cite{GihSko}. This fact, in conjunction with the uniqueness property of the mild solution, enables us to write
\begin{equation}\label{mild1}
u(t) = S(t-\tau) u(\tau) + \int_{\tau}^{t} S(t-s) f(u(s)) ds + \int_{\tau}^{t} S(t-s) \sigma(u(s)) dW(s).
\end{equation}
Furthermore,
\[
\E \|u(t,\omega)\|_{L^2(\R^d)}^2 = \int_{\{\omega: \|u\|_{L^2(\R^d)} \leq M+1\}} \|u(t,\omega)\|_{L^2(\R^d)}^2 dP(\omega) + \int_{C}  \|u(t,\omega)\|_{L^2(\R^d)}^2 dP(\omega)
\]
\[
\leq (M+1)^2 + \int_{C} \|u(t,\omega)\|^2 dP(\omega)
\]
It follows from the condition (ii) and (\ref{mild1}) that for $\omega \in C$
\[
u(t, \omega) = S(t-\tau) u(\tau, \omega) + \int_{\tau}^{t} S(t - \tau) \sigma(u(s)) d W(s)
\]
then
\[
\int_{C} \|u(t,\omega)\|^2_{L^2(\R^d)} d P(\omega) \leq 2 \left[\int_{C}\|S(t-\tau)u(\tau)\|_{L^2(\R^d)}^2 dP(\omega) + \int_{C}\left\|\int_{\tau}^{t} S(t-s) \sigma(u(s)) dW(s)\right\|_{L^2(\R^d)}^2\right]
\]
\begin{equation}\label{xxx}
\leq 2 \left[\E\|S(t-\tau)u(\tau)\|_{L^2(\R^d)}^2  + \E \left\|\int_{\tau}^{t} S(t-s) \sigma(u(s)) dW(s)\right\|_{L^2(\R^d)}^2\right]
\end{equation}
The first term is bounded by using the contraction property of $S(t)$ in
$L^2(\R^d)$:
\[
\E\|S(t-\tau)u(\tau)\|_{L^2(\R^d)}^2 \leq \E \|u(\tau)\|_{L^2(\R^d)}^2 = (M+1)^2
\]
For the second term in (\ref{xxx}), we compute,
\[
\E \left\| \int_{\tau}^{t} S(t-s) \sigma(u(s)) d W(s)\right\|_{L^2(\R^d)}^2 \leq \E\left(\sup_{0 \leq \nu \leq t} \left\| \int_{\nu}^{t} S(t-s) \sigma(u(s)) d W(s)\right\|_{L^2(\R^d)}^2\right)
\]
\[
\leq 2 \E \left\| \int_{0}^{t} S(t-s) \sigma(u(s)) d W(s)\right\|^2_{L^2(\R^d)} + 2 \E \left( \sup_{0 \leq \nu \leq t} \left\| \int_{0}^{\nu} S(t-s) \sigma(u(s)) d W(s)\right\|^2_{L^2(\R^d)}\right)
\]
By the following Doob's Inequality for martingales,
\[
\E\left(
\sup_{0 \leq \nu \leq t}\left|\int_{0}^{\nu} \sum_{k=1}^{\infty} g_k(s) d \beta_{k}(s)\right|^2\right)
\leq 4 \sum_{k=1}^{\infty} \E \int_{0}^{t} |g_k(s)|^2 ds,
\]
we have
\[
\E \left(\sup_{0\leq \nu \leq t}  \left\| \int_{0}^{\nu} S(t-s) \sigma (u(s)) d W(s)\right\|^2_{L^2(\R^d)} \right)
\]
\[
= \E \left(\sup_{0\leq \nu \leq t} \left(\int_{\R^d} \left| \int_{0}^{\nu} \sum_{k=1}^{\infty} \sqrt{a_k} \int_{\R^d}G(t-s,x-y)\sigma(y,u(s,y,\omega)) e_k(y) dy d \beta_k(s)\right|^2 dx\right)\right)
\]
\begin{equation}\label{bnd_thm2}
\leq 4 \int_{\R^d}\left( \E \int_{0}^{t}  \sum_{k=1}^{\infty} a_k \left( \int_{\R^d}G(t-s,x-y)\sigma(y,u(s,y,\omega)) e_k(y) dy\right)^2 ds \right) dx
\end{equation}
Similarly to the proof of Theorem \ref{thm1}, we split $\int_{0}^{t} =\int_{0}^{t-1} + \int_{t-1}^{t}$. Then
\[
\int_{\R^d}\left( \E \int_{t-1}^{t}  \sum_{k=1}^{\infty} a_k \left( \int_{\R^d}G(t-s,x-y)\sigma(y,u(s,y,\omega)) e_k(y) dy\right)^2 ds \right)  dx
\]
\[
\leq \sum_{k=1}^{\infty} a_k \int_{t-1}^{t} \int_{\R^d} \left(\int_{\R^d}G(t-s,x-y) \, dy\right)  \left(\int_{\R^d}G(t-s,x-y) \psi^2(y) e_k^2 (y) dy\right) dx ds
\]
\[
= \sum_{k=1}^{\infty} a_k \int_{t-1}^{t} \int_{\R^d}
\int_{\R^d}G(t-s,x-y) \psi^2(y) e_k^2 (y) dy dx ds
\]
\[
= \sum_{k=1}^{\infty} a_k \int_{t-1}^{t}
\int_{\R^d}G(t-s,x-y) dx \int_{\R^d}\psi^2(y) e_k^2 (y) dy ds
\leq \sum_{k=1}^{\infty} a_k \|\psi\|_{L^2(\R^d)}^2 < \infty.
\]
Next,
\[
\int_{\R^d}\left( \E \int_{0}^{t-1}  \sum_{k=1}^{\infty} a_k \left( \int_{\R^d}G(t-s,x-y)\sigma(y,u(s,y,\omega)) e_k(y) dy\right)^2 ds \right) dx
\]
\[
\leq \sum_{k=1}^{\infty} a_k \int_{0}^{t-1} \int_{\R^d} \left( \int_{\R^d} G^2(t-s,x-y) e_k^2(y) \, dy \right)\left( \int_{\R^d} \psi^2(y) \, dy\right) dx ds
\]
\[
\leq \sum_{k=1}^{\infty} a_k \|\psi\|^2_{L^2(\R^d)}
\int_{0}^{t-1} \int_{\R^d} G^2(t-s,z) dz dt
\leq \sum_{k=1}^{\infty} a_k \|\psi\|^2_{L^2(\R^d)}
\int_{0}^{t-1} \frac{1}{(t-s)^{\frac{d}{2}}}dt
< \infty.
\]
The above complete the proof of Theorem \ref{thm2}.
\end{proof}

\section{Proof of Theorem 3}\label{SecInvMeasUpperHalf}
In this section, we analyze the equation (\ref{eqn1semi}).
We follow the notations immediately after \eqref{eqn1semi} on page
\pageref{eqn1semi.notation}. For the proof, we introduce the following
infinite dimensional Wiener process:
\begin{equation}\label{Wiener_all t}
W(t,x) = \sum_{k=1}^{\infty}\sqrt{a_k} \beta_k(t) e_k(x)
\end{equation}
where $e_k(x)$'s satisfy (\ref{basis}) and we also require
\[
a:= \sum_{k=1}^{\infty} a_k < \infty
\]
In contrast with the previous section,
the Wiener process in this section is defined for all $t \in \R$. This can be constructed by the following formula:
\[
\beta_k(t) = \left\{
\begin{array}{ll}
\beta_k^{(1)}(t),&\text{for}\,\,\,t \geq 0\\
\beta_k^{(2)}(-t),&\text{for}\,\,\,t \leq 0
\end{array}
\right.,
\]
where $\beta_k^{(1)}$ and $\beta^{(2)}_k$ are independent standard
one dimensional Wiener processes.
Also, let
\[
\F_t := \bigcup\{\beta_k(v) - \beta_k(u): u \leq v \leq t, k \geq 1\}
\]
be the $\sigma$-algebra generated by $\{\beta_k(v) - \beta_k(u): u \leq v \leq t, k \geq 1\}$.



Our proof heavily relies on the spectral properties of the operator $A$
in some weighted space. These are described next.

\subsection{Eigenvalue problem for $A$.}
In the case $d=1$, consider the weight function $\rho = e^{-x^2}$.
We then have the following problem for determining the spectrum: find all
$\mu \in \R$ and $w \in H = L^2_{\rho}(\R^+)$ such that
\begin{equation}\label{eig1}
e^{x^2}\frac{d}{d x}\left(\frac{d w}{d x} e^{-x^2}\right) = \mu w, \  x > 0;
\end{equation}
satisfying
\begin{equation}\label{BC at infinity}
\int_{0}^{\infty} w^2 e^{-x^2} dx < \infty
\end{equation}
and
\begin{equation}\label{BC0}
w(0) = 0.
\end{equation}
The problem (\ref{eig1}) is a well known problem for harmonic oscillator
\cite[p. 218-219]{SamPerPar12}.
It has a nonzero solution satisfying (\ref{BC at infinity}) only for $\mu = - 2 n, n = 0, 1, 2, ... $. The solutions are the Hermite polynomials $w_n = H_n(x)$. Moreover, the condition (\ref{BC0}) implies that $n$ must be odd. Therefore, the eigenvalues of (\ref{eig1}) are $\mu = 2 - 4p, p = 1, 2, 3,...$

If $d > 1$, the eigenvalue problem reads as
\begin{equation}\label{eig>1}
\Delta w - 2 (\nabla w, x) = \mu w,
\end{equation}
subject to
\begin{equation}
\int_{\R^d_+} w^2 e^{-|x|^2} dx < \infty
\end{equation}
and
\begin{equation}
w(x_1,...,x_{d-1}, 0) = 0.
\end{equation}
We proceed with looking for the solutions of (\ref{eig>1}) using
separation of variables,
\[
w(x_1,x_2,...,x_d) = w_1(x_1) w_2(x_2) ... w_d(x_d),
\]
with $w_i$ solving
\begin{equation}\label{eigd-1}
e^{x_i^2}\frac{d}{d x_i}\left(\frac{d w_i}{d x_i} e^{-x_i^2}\right) = \lambda_i w_i, \ i = 1,...,d
\end{equation}
subjects to
\begin{equation}\label{BCd-1}
\int_{\R} w_i^2(x) e^{-x^2} d x < \infty, i = 1,..., d-1;
\end{equation}
and
\begin{equation}\label{BCd}
\int_{0}^{\infty} w_d^2(x) e^{-x^2} d x < \infty, \ w_d(0) = 0.
\end{equation}
It follows from the condition (\ref{BCd-1}) that
for $i = 1,...,d-1$, we have
\[
\lambda_i = -2 p, p = 0, 1, 2,...
\]
while due to (\ref{BCd})
\[
\lambda_d = -2 - 4p, p=0, 1, 2, ...
\]
An arbitrary eigenvalue $\mu$ of (\ref{eig>1}) satisfies $\mu = \lambda_1 + ... + \lambda_d$. In particular, the largest eigenvalue of
(\ref{eig1}) is given by $\mu_1 = -2$ (which corresponds to $\lambda_1 = ... = \lambda_{d-1} = 0, \lambda_d = -2$).

With the above, we have the following technical lemmas.

\begin{lem}\label{lem1}
Let $S(t):H \to H$ be a semigroup generated by $A$, i.e. $S(t)u_0(x): = u(t,x)$, where $u(t,x)$ solves
\begin{equation}\label{heat}
\begin{cases}
u_t(t,x) = A u(t,x)\\
u(0,x) = u_0(x).
\end{cases}
\end{equation}
Then
\begin{equation}\label{exp_est}
\|S(t)u_0\|_{H} \leq e^{-2 t}\|u_0\|_{H}
\end{equation}
\end{lem}
\begin{proof}
Let $0 > \mu_1 > \mu_2 \geq \mu_3 \geq...$, with $\mu_1 = -2$, be the eigenvalues of $A$,  and let $\{\f_k(x), k \geq 1\} \in H$ be the corresponding orthonormal eigenbasis. We have the following representations for  $u_0 \in H$ and $u(t,x) \in H$:
\[
u_0(x) = \sum_{k=1}^{\infty} c^0_k \f_k(x)
\]
and
\[
u(t,x) = \sum_{k=1}^{\infty} c_k(t) \f_k(x)
\]
It follows from (\ref{heat}) that
\[
\sum_{k=1}^{\infty} c_k^{'}(t) \f_k(x) = \sum_{k=1}^{\infty} c_k(t) A \f_k(x) = \sum_{k=1}^{\infty} \mu_k c_k(t) \f_k(x)
\]
Thus
\[
c_k(t) = c^{0}_k e^{\mu_k t}
\]
Hence
\[
\|u(t,x)\|^2_{H} = \sum_{k=1}^{\infty}c_k^{2}(t) = e^{-4 t} \sum_{k=1}^{\infty} e^{(2\mu_k + 4)t} (c^{0}_{k})^{2} \leq e^{-4 t} \|u_0(x)\|_{H}^2
\]
concluding the proof.
\end{proof}
\begin{lem}\label{lem2}
For any $u \in H$, we have
\begin{equation}\label{norm1}
\E \left\|\int_{t_0}^{t} S(t-s) \sigma(u(s)) d W(s) \right\|^2_{H}  \leq \sum_{k=1}^{\infty} a_k \int_{t_0}^{t} e^{-4(t-s)} \E \|\sigma(u(s))\|_H^2 ds
\end{equation}
\end{lem}
\begin{proof} It is a consequence of the following computation.
\[
\E \left\|\int_{t_0}^{t} S(t-s) \sigma(u(s)) d W(s) \right\|^2_{H}  = \E \left\|\sum_{k=1}^{\infty} \sqrt{a_k} \int_{t_0}^{t} S(t-s) \sigma(u(s)) e_k(x) d \beta_k(s) \right\|^2_{H}
\]
\[
= \int_G \E \left( \sum_{k=1}^{\infty} \sqrt{a_k} \int_{t_0}^{t} S(t-s) \sigma(u(s)) e_k(x) d \beta_k(s)\right)^2 \rho(x) \, dx
\]
\[
= \int_G \sum_{k=1}^{\infty} a_k \E \left(\int_{t_0}^{t} S(t-s) \sigma(u(s)) e_k(x) d \beta_k(s)\right)^2 \rho(x) \, dx
\]
\[
 = \sum_{k=1}^{\infty} a_k  \int_G \int_{t_0}^{t} \E\left(S(t-s) \sigma(u(s)) e_k(x)\right)^2 ds \, \rho(x) \,  dx
 \]
 \[
 = \sum_{k=1}^{\infty} a_k \int_{t_0}^{t} \E\left\|S(t-s) \sigma(u(s)) e_k(x)\right\|_H^2 ds \leq \sum_{k=1}^{\infty} a_k \int_{t_0}^{t} e^{-4(t-s)} \E \|\sigma(u(s))\|_H^2 ds.
\]
\end{proof}
\subsection{Well-posedness for Equation \eqref{eqn1semi}}
Here we show the existence and uniqueness of the solution for
\eqref{eqn1semi}. For simplicity again, we omit the $x$ variable in
$f$ and $\sigma$.
\begin{thm}\label{thm3exist}
Assume that $f$ and $\sigma$ satisfy (\ref{Lip}) and (\ref{linf}). Then, for given $u_0(x) \in H$, there exists a unique mild solution of (\ref{eqn1semi}) (see Definition \ref{defMild0}).
\end{thm}

\begin{proof}
Write the integral relation (\ref{eqn1semi}) as
\begin{equation}\label{fix}
u(t,x) = \Psi[u(t,x)]
\end{equation}
where
\[
\Psi[v(t,x)] := S(t) u_0(x) + \int_{0}^{t} S(t-s) f(v(s)) ds + \sum_{k=1}^{\infty} \sqrt{a_k} \int_{0}^{t}S(t-s) \sigma(v(s)) e_k d \beta_k(s)
\]
For $T>0$, let
\[
B:=\{v \in H  \text{ is } \mathcal{F}_{t}
\text{ measurable for $\forall t \in [0, T]$},\,\,\sup_{t \in [0,T]} \E \|v(t)\|^2_{H}< \infty \}
\]
and
\[
\|v\|_{B}^2 := \sup_{t \in [0,T]} \E\|v(t)\|^2_{H}
\]
We will establish the contraction property of $\Psi$:
for $T$ sufficiently small, it holds that
\begin{itemize}
\item[ (i) ] $\Psi: B \to B$;\\
\item[(ii)] $\|\Psi(v_1) - \Psi(v_2)\|_{B} \leq \gamma \|v_1-v_2\|_{B} \text{ for some } 0< \gamma<1$
\end{itemize}

To show (i), for $u \in B$ and $t \in [0,T]$, we have:
\[
\|\Psi [u] \|_{B} = \sup_{t \in [0,T]}\|\Psi[u]\|^2_{H} \leq 3\left(\sup_{t \in [0,T]}\|S(t)u_0(x)\|^2_{H} + \sup_{t \in [0,T]} I_1 + \sup_{t \in [0,T]} I_2\right)
\]
where
\[
I_1:= \E \left\|\int_{0}^{t} S(t-s) f(u(s)) ds \right\|^2_H
\,\,\,\text{and}\,\,\,
I_2:= \E \left\|\sum_{k=1}^{\infty} \sqrt{a_k} \int_{0}^{t}S(t-s) \sigma(u(s)) e_k d \beta_k(s)\right\|^2_{H}
\]
First, note that by Lemma \ref{lem1},
\[
\sup_{t \in [0,T]} \|S(t) u_0(x)\|^2_{H} \leq \sup_{t \in [0,T]} e^{-4t} \|u_0\|^2_{H} < \infty
\]
We next proceed with estimating $I_1$ and $I_2$.
\[
I_1 \leq \E \left(\int_{0}^{t}\left\|S(t-s)f(u(s))\right\|_{H} ds\right)^2 \leq \E \left(\int_{0}^{t} e^{-2(t-s)} \|f(u(s))\|_{H} ds\right)^2
\]
\[
\leq \int_{0}^{t} e^{-2(t-s)} \int_{0}^{t} e^{-2(t-s)} \E \|f(u(s))\|_{H} ds  \leq C (1 + \|u\|^2_{B}).
\]
Hence
\[
\sup_{t \in [0,T]} I_1 \leq C(1 + \|u\|^2_{B}) < \infty
\]
Similarly, using  (\ref{basis})
\[
I_2 = \int_{0}^{t} \sum_{k=1}^{\infty} a_k \E \|S(t-s) \sigma(u(s,x)) e_k(x) \|_{H}^2 ds
\]
\[
\leq \sum_{k=1}^{\infty} a_k \int_{0}^{t} e^{-4(t-s)} \E \| \sigma(u(s)) e_k(x) \|_{H}^2 ds
\]
\[
\leq \sum_{k=1}^{\infty} a_k \int_{0}^{t} e^{-4(t-s)} \E \| \sigma(u(s)) \|_{H}^2 ds \leq C (1 + \|u\|^2_{B})
\]
Thus (i) follows.

To show (ii), let $u_1$ and $u_2$ be arbitrary elements in $H$. For $t \in [0,T]$,
we have
\[
\E \|\Psi(u_1) - \Psi(u_2)\|^2_{H} \leq 2(J_1 + J_2)
\]
where
\[
J_1 := \E \left\|\int_{0}^{t} S(t-s)(f(u_1(s)) - f(u_2(s))) ds\right\|_H^2
\]
and
\[
J_2 :=  \E \left\| \sum_{k=1}^{\infty} \sqrt{a_k} \int_{0}^{t} S(t-s) (\sigma(u_1(s)) - \sigma(u_2(s))) e_k d \beta_k(s)\right\|^2_H
\]
For $t \in [0,T]$, we have
\[
J_1 \leq \E \left(\int_{0}^{t} \| S(t-s)(f(u_1(s)) - f(u_2(s)))\|_{H} ds\right)^2
\]
\[
\leq \E \left(\int_{0}^{t} e^{-2(t-s)} \|f(u_1(s)) - f(u_2(s))\|_{H} ds\right)^2\]
\[
\leq \int_{0}^{t} e^{-2(t-s)} ds \int_{0}^{t} e^{-2(t-s)}  \E\|f(u_1(s)) - f(u_2(s))\|^2_{H} ds  \leq  \frac{L^2 T}{2} \|u_1 - u_2\|^2_{B}
\]
Similarly,
\[
J_2 \leq L^2  \sum_{k=1}^{\infty} a_k \int_{0}^{t}  e^{-4(t-s)} \E\|(u_1(s) - u_2(s)) e_k\|^2_{H} ds
\leq \frac{L^2 T}{4}   \sum_{k=1}^{\infty} a_k   \|u_1 - u_2\|^2_{B}
\]
Thus, we have
\[
\|\Psi(u_1) - \Psi(u_2)\|_{B}^2 \leq \gamma \|u_1 - u_2\|^2_{B}
\]
where
\[
\gamma =\frac{L^2 T}{2}\left(1 + \frac{1}{2}   \sum_{k=1}^{\infty} a_k\right) < 1
\]
for sufficiently small $T>0$. Therefore, $\Psi$ is a contraction
which implies a unique fixed point for the operator $\Phi$ leading
to a mild solution of (\ref{eqn1}) on $[0,T]$.

Repeating the above procedure for the intervals $[T,2T]$, $[2T,3T]$
$\ldots$, we get the existence result on $[0, \infty)$.
\end{proof}

Next, we will construct and analyze solutions of
(\ref{eqn1semi}) defined {\em for all $t \in \R$}. First we introduce
the following definition.
\begin{definition}\label{defMild}
We say that an $H$-valued process $u(t)$ is a mild solution of (\ref{eqn1semi}) on $\R^1$ if
\begin{enumerate}
\item for $\forall t \in \R$, $u(t)$ is $\mathcal{F}_t$-measurable;
\item $u(t)$ is continuous almost surely in $t \in \R$ with respect to $H$-norm;
\item $\forall t \in \R$, $\E\|u(t)\|_H^2 < \infty$
\item for all $-\infty < t_0 < t < \infty$ with probability 1 we have
\begin{equation}\label{mildall}
u(t) = S(t-t_0)u(t_0) + \int_{t_0}^{t} S(t-s) f(u(s)) ds + \int_{t_0}^{t} S(t-s) \sigma(u(s)) dW(s)
\end{equation}
\end{enumerate}
\end{definition}

The proof of Theorem 3 is divided into its linear
and nonlinear versions.

\subsection{Proof of Theorem 3 -- Linear Version.}
Let $\mathcal{B}$ be the class of $H$-valued, $\mathcal{F}_t$-measurable random processes $\xi(t)$ defined on $\R^1$ such that
\begin{equation}
\sup_{t \in \R^1} E \|\xi(t)\|^2_{H} < \infty
\end{equation}
For $\f(t)$ and $\alpha(t)$ in $\mathcal{B}$ consider
\begin{equation}\label{lin}
du = (A u + \alpha(t)) dt + \f(t) d W(t)
\end{equation}
\begin{definition}\label{defExpStab}
A solution $u^*$ is exponentially stable in mean square
if there exist $K>0$ and $\gamma>0$ such that for any $t_0$ and any
other solution $\eta(t)$, with $\mathcal{F}_{t_0}$ measurable $\eta(t_0)$ and $E\|\eta(t_0)\|_{H}^2 < \infty$, we have
\[
\E \|u^*(t) - \eta(t)\|_H^2 \leq K e^{-\gamma(t-t_0)} \E \|u^{*}(t_0) - \eta(t_0) \|^2_{H}
\]
for $t \geq t_0$.
\end{definition}
\begin{thm}\label{thm4lin}
The equation (\ref{lin}) has a unique solution $u^{*}$ in the sense of the Definition \ref{defMild}. This solution is in $\mathcal{B}$ and is exponentially stable in the sense of Definition \ref{defExpStab}.
\end{thm}
\begin{proof}
Define
\begin{equation}\label{lin_explicit}
u^{*}(t):= \int_{-\infty}^{t} S(t-s)\alpha(s) ds + \int_{-\infty}^{t}S(t-s) \f(s) d W(s)
\end{equation}
We start with showing that the function given by (\ref{lin_explicit}) is well-defined in the sense that the improper integrals are convergent. Let
\begin{equation}\label{sequence_xi}
\xi_n (t) := \int_{-n}^{t} S(t-s) \alpha(s) ds
\end{equation}
\begin{equation}\label{sequence_zeta}
\zeta_n (t) := \int_{-n}^{t} S(t-s) \f(s) d W(s)
\end{equation}
For $n>m$, we have
\[
\E\|\xi_n(t) - \xi_m(t)\|_{H}^2 \leq \E\left(\int_{-n}^{-m} \|S(t-s) \a(s)\|_H ds\right)^2 \leq  \E \left( \int_{-n}^{-m} e^{-2(t-s)} \|\alpha(s)\|_H ds\right)^2
\]
\[
\leq \int_{-n}^{-m} e^{-2(t-s)} ds \cdot \int_{-n}^{-m} e^{-2(t-s)} \E \|\alpha(s)\|_H^2 ds \leq \sup_{t \in \R}\E \|\alpha(t)\|_H^2 \cdot \left(\int_{-n}^{-m} e^{-2(t-s)} ds\right)^2
\]
which can be made as small as possible as $n,m \to \infty.$
Thus for all $t \in \R$ the sequence (\ref{sequence_xi}) is a Cauchy sequence.

Similarly, using Lemma \ref{lem2}, we have
\[
\E \|\zeta_n(t) - \zeta_m(t)\|_H^2 = \E \left\|\int_{-n}^{-m} S(t-s) \f(s) d W(s)\right\|_H^2 \leq  \sum_{k=1}^{\infty} a_k \int_{-n}^{-m} e^{-2(t-s)} \, ds \sup_{t \in \R} \E \|\f(t)\|_H^2
\]
which is again uniformly small for all large $n$ amd $m$. Thus
$\left\{\zeta_n\right\}_n$ is also a Cauchy sequence. The above show that
the process given by (\ref{lin_explicit}) is well defined.

We will show that this process is the solution in the sense of Definition
\ref{defMild}. First, we note that $u^{*}(t)$ is $\F_t$-measurable. Furthermore, the continuity of $u^*$ in time with probability 1 follows from the factorization formula for the stochastic integrals \cite[Theorem 5.2.5]{DapZab96}. Next we show that
\begin{equation}\label{bnd_ustar}
\sup_{t \in \R} \E \|u^{*}(t)\|_{H}^2 < \infty
\end{equation}
From (\ref{lin_explicit}), we have
\[
\E \left\|\int_{-\infty}^{t} S(t-s) \a(s) ds\right\|_H^2 \leq \E \left(\int_{-\infty}^{t} \|S(t-s)\a(s)\|_H ds \right)^2
\]
\[
\leq \int_{-\infty}^{t} e^{-2(t-s)} ds \int_{-\infty}^{t} e^{-2(t-s)} ds \sup_{t \in \R} \E\|\a(t)\|_H^2 = \frac{1}{4}\sup_{t \in \R} \E\|\a(t)\|_H^2
\]
as well as
\[
\E \left\|\int_{-\infty}^{t} S(t-s) \f(s) d W(s)\right\|_H^2 \leq  \sum_{k=1}^{\infty} a_k \int_{-\infty}^{t} e^{-4(t-s)} \sup_{t \in \R} \E \|\f(t)\|_H^2 < \infty.
\]
Thus (\ref{bnd_ustar}) holds.

Finally, since
\[
u^{*}(t_0) = \int_{-\infty}^{t_0} S(t_0-s) \a(s) ds + \int_{-\infty}^{t_0} S(t_0-s) \f(s) dW(s)
\]
we compute:
\[
u^{*}(t) = \int_{-\infty}^{t} S(t-s) \a(s) ds + \int_{-\infty}^{t} S(t-s) \f(s) dW(s)
\]
\[
=
\int_{-\infty}^{t_0} S(t-s) \a(s) ds +
\int_{-\infty}^{t_0} S(t-s) \f(s) dW(s)
\]
\[
+ \int_{t_0}^t S(t-s) \a(s) ds +
\int_{t_0}^t S(t-s) \f(s) dW(s)
\]
\[
=
\int_{-\infty}^{t_0} S(t-t_0)S(t_0-s) \a(s) ds +
\int_{-\infty}^{t_0} S(t-t_0)S(t_0-s) \f(s) dW(s)
\]
\[
+ \int_{t_0}^t S(t-s) \a(s) ds +
\int_{t_0}^t S(t-s) \f(s) dW(s)
\]
\[
= S(t-t_0) u^{*}(t_0) + \int_{t_0}^{t} S(t-s) \a(s) ds + \int_{t_0}^{t} S(t-s) \f(s) dW(s).
\]
Hence $u^{*}$ is a solution in the sense of Definition \ref{defMild}.

To show the exponential stability of $u^{*}$ (in the sense of
Definition \ref{defExpStab}),
let $\eta(t)$ be another solution of (\ref{lin}), such that $\E \|\eta(t_0)\|_H^2 < \infty$. Then
\[
\eta(t) = S(t-t_0) \eta(t_0) + \int_{t_0}^{t}S(t-s) \a(s) ds + \int_{t_0}^{t} S(t-s) \f(s) d W(s),
\]
and thus
\[
\E \|u^{*}(t) - \eta(t)\|_H^2
=
\E \|S(t-t_0)(u^{*}(t_0) - \eta(t_0))\|_H^2
\leq e^{-4(t-t_0)}\E\|u^*(t_0) - \eta(t_0)\|^2_H
\]
which implies the stability of $u^*$.

Finally, we show the uniqueness of $u^{*}$. Let $u_0$ be another solution, such that
\[
\sup_{t \in \R} \E \|u_0(t)\|_{H}^2 < \infty
\]
Then $z(t) = u^{*}(t) - u_0(t)$ satisfies
\[
\E\|z(t)\|^2 \leq e^{-4(t-\tau)} \E\|z(\tau)\|_H^2
\]
for arbitrary $\tau \leq t$. Clearly, $\sup_{t \in \R} \E\|z(t)\|^2_{H} \leq e^{-4(t-\tau)} C$ for some $C>0$. Letting $\tau \to - \infty$, we have $\E \|z(t)\|^2 = 0$ for all $t \in \R$. Therefore,
\[
\P\left(u_0(t) \neq u^{*}(t)\right) = 0, \ \forall t \in \R.
\]
Since the processes $u_0$ and $u^*$ are continuous in time with probability 1, then
\[
\P\left(\sup_{t \in \R}\|u_0(t) - u^*(t)\|_H>0\right) = 0.
\]
\end{proof}

Now we are ready to prove Theorem \ref{thm5nonlin}.

\subsection{Proof of Theorem \ref{thm5nonlin} - Nonlinear Version}
 \begin{proof}
Suppose the constant $L$ in (\ref{Lip}) satisfies
\begin{equation}\label{smallness}
L^2 + L^2 \sum_{k=1}^{\infty} a_k < 1
\end{equation}
\begin{equation}\label{small2}
\frac{L^2}{2} + L^2 \sum_{k=1}^{\infty} a_k < \frac{2}{3}.
\end{equation}
The idea of the proof is to construct a sequence of approximations which converges to the solution $u^*(t,x)$. Let $u_0 \equiv 0$. For $n \geq 0$, define $u_{n+1}(t,x)$ as
\begin{equation}\label{iter}
d u_{n+1} = (A u_{n+1} + f(x,u_n))\, dt + \sigma(x,u_n) dW(t)
\end{equation}
The equation (\ref{iter}) satisfies the conditions of the Theorem \ref{thm4lin}, since
\[
\sup_{t \in \R} \E \|f(x,u_n(t,x))\|^2_H \leq C  \sup_{t \in \R} \E \int_{G}(1+|u_n(t,x)|^2) e^{-|x|^2} \, dx < \infty.
\]
for some $C>0$. The bound for $\sigma(x, u_n)$ is obtained analogously. Therefore, by Theorem \ref{thm4lin}, we can find the unique $u_{n+1}(t,x)$ satisfying
\[
\sup_{t \in \R} \E \|u_{n+1}\|^2_H < \infty.
\]

First, we show that $\sup_{t \in \R} \E \|u_n\|^2_H$ has a bound which is {\it independent of $n$}. To this end, $u_{n+1}$ has the presentation
\[
u_{n+1}(t) = \int_{-\infty}^{t} S(t-s) f(u_n(s)) ds + \int_{-\infty}^{t}S(t-s)\sigma(u_n(s)) dW(s):= I_1 + I_2
\]
thus
\[
E \|u_{n+1}(t)\|_H^2 \leq 2 \E \|I_1\|_{H}^2 + 2 E \|I_2\|_{H}^2.
\]
We now estimate each term separately:
\[
 \E \|I_1\|_{H}^2 = \E \left\|\int_{-\infty}^{t}S(t-s) f(u_n(s)) ds \right\|_H^2\]
\[
\leq 2 \, \E \left\|\int_{-\infty}^{t}S(t-s) f(0) ds \right\|_H^2 +
2 \, \E \left\|\int_{-\infty}^{t}S(t-s) [f(u_n(s)) - f(0)] ds \right\|_H^2
\]
\[ \leq C_0 \int_{-\infty}^{t} e^{-2(t-s)} ds  + L^2 \E  \int_{-\infty}^{t} e^{-2(t-s)} \|u_n(s)\|_H^2 ds
\leq C_0 + \frac{L^2}{2}\sup_{t \in \R} \E \|u_n(t)\|_H^2
\]
Applying Lemma \ref{lem2}, we proceed with a similar estimate for $I_2$:
\[
\E \|I_2\|_{H}^2 =  \E \left\| \int_{-\infty}^{t}S(t-s)\sigma(u_n(s)) dW(s)\right\|_H^2
\]
\[
\leq 2 \, \E \left\|\int_{-\infty}^{t} S(t-s)\sigma(0) dW(s)\right\|_H^2
+ 2 \, \E \left\|\int_{-\infty}^{t} S(t-s)[\sigma(u_n(s)) - \sigma(0)] dW(s)\right\|_H^2
\]
\[
\leq C_1 + 2  \sum_{k=1}^{\infty} a_k \int_{-\infty}^{t} e^{-4(t-s)} \E \|u_n(s)\|^2_H \, ds
\leq C_1 + \sum_{k=1}^{\infty} a_k \frac{L^2}{2} \sup_{t \in \R} \E \|u_n(t)\|_H^2
\]
so that we have
\[
\sup_{t \in \R} E \|u_{n+1}(t)\|_H^2 \leq C_2 + L^2(1 +  \sum_{k=1}^{\infty} a_k)\sup_{t \in \R} \E \|u_n(t)\|_H^2
\]
where $C_2 = 2 C_0 + 2 C_1$ does not depend on $L$. Hence, if $L^2(1 +  \sum_{k=1}^{\infty} a_k) < 1$ (condition (\ref{smallness})), we have a bound for $\sup_{t \in \R} E \|u_{n}(t)\|_H^2 $ which is independent of $n$:
\begin{equation}\label{unif_bound}
\sup_{t \in \R} E \|u_{n}(t)\|_H^2 \leq \frac{C_2}{1 - L^2(1 +  \sum_{k=1}^{\infty} a_k)}
\end{equation}
The bound (\ref{unif_bound}) follows from the fact that if a nonnegative numerical sequence $\{x_n, n\geq 1\}$ satisfies
\[
x_{n+1} \leq a + b x_n
\]
with $b<1$, then $x_n \leq \frac{a}{1-b}$.

Second, we establish that $u_n$ is convergent.
\begin{multline*}
u_{n+1}(t) - u_{n}(t)=  \int_{-\infty}^{t} S(t-s) [f(u_n(s)) - f(u_{n-1}(s))] ds + \\ + \int_{-\infty}^{t}S(t-s)[\sigma(u_n(s)) - \sigma(u_{n-1}(s))] dW(s):=
 J_1 + J_2.
\end{multline*}
Thus
\[
\E \|u_{n+1}(t) - u_{n}(t)\|_H^2 \leq 2 \E \, \|J_1\|^2_{H} + 2 \E \, \|J_2\|^2_H.
\]
Estimating the first term, we have
\[
\|J_1\|^2_{H} = \E \left\|\int_{-\infty}^{t}S(t-s) [f(u_n(s)) - f(u_{n-1}(s))] ds \right\|_H^2
\]
\[
\leq \frac{L^2}{2}  \int_{-\infty}^{t} e^{-2(t-s)} \E \|u_n(s) - u_{n-1}(s)\|_H^2 ds
\leq \frac{L^2}{4} \sup_{t \in \R} \E\|u_n(t) - u_{n-1}(t)\|_H^2.
\]
Using Lemma \ref{lem2} again, we have
\[
\|J_2\|^2_{H} \leq L^2 \sum_{k=1}^{\infty} a_k \int_{-\infty}^{t} e^{-4(t-s)}\E \|u_n(s) - u_{n-1}(s)\|_H^2 ds
\leq \frac{L^2}{4} \sum_{k=1}^{\infty} a_k \sup_{t \in \R} \E\|u_n(t) - u_{n-1}(t)\|_H^2.
\]
Therefore,
\begin{equation}\label{iter_diff}
\sup_{t \in \R} \E \|u_{n+1}(t) - u_{n}(t)\|_H^2 \leq \frac{L^2}{2} \left(1 +  \sum_{k=1}^{\infty} a_k  \right) \sup_{t \in \R} \E\|u_n(t) - u_{n-1}(t)\|_H^2.
\end{equation}
where, due to (\ref{smallness}),
\[
\frac{L^2}{2} \left(1 + \sum_{k=1}^{\infty} a_k\right) < \frac{1}{2}.
\]
Iterating (\ref{iter_diff}), we get
\[
\sup_{t \in \R} \E \|u_{n+1}(t) - u_{n}(t)\|_H^2 \leq \frac{C}{2^n}
\]
for some positive constant $C$. Therefore, $\forall n, m \geq 1$
\[
\sup_{t \in \R} \sqrt{\E \|u_n(t) - u_{m}(t)\|^2_{H}} = \sup_{t \in \R} \sqrt{\E \left\|\sum_{i=m}^{n}(u_{i+1}(t) - u_{i}(t))\right\|^2_{H}}
\]
\[
\leq \sum_{i=m}^{n}\sqrt{\sup_{t \in \R} \E \|u_{i+1}(t) - u_{i}(t)\|^2_{H}} \to 0, \,\,\text{as}\,\, n,m \to \infty,
\]
and thus $u_n(t)$ is a Cauchy sequence. Consequently, there is a limiting function $u^{*}(t, \cdot) \in H$ such that
\[
\sup_{t \in \R} \E \|u_n(t) - u^{*}(t)\|^2_H \to 0, \ n \to \infty.
\]
Using (\ref{unif_bound}), it follows from Fatou's Lemma that
\[
\sup_{t \in \R} \E \|u^{*}\|^2_{H}  \leq \frac{C_2}{1 - L^2(1 + \sum_{k=1}^{\infty} a_k)}
\]
The function $u^{*}(t)$ is $\F_{t}$-measurable as a limit of $\F_t$-measurable processes.

Third, we show that $u^{*}$ solves the equation (\ref{eqn1semi}). To this end, we need to pass to the limit in the identity
\begin{equation}\label{identity}
u_{n+1}(t) = S(t-t_0) u_{n+1}(t) + \int_{t_0}^{t} S(t-t_0) f(u_n(s)) ds + \int_{t_0}^{t} S(t-s)\sigma(u_n(s)) dW(s)
\end{equation}
Using Markov's inequality, $\forall \ve > 0$
\[
\sup_{t \in \R} \P\{\|u_n(t) - u^{*}(t)\|_{H} > \ve\} \leq \frac{\sup_{t \in \R} \E\|u_n(t) - u^{*}(t)\|^2}{\ve^2} \to 0, \ n \to \infty.
\]
So $u_n(t) \to u^{*}(t)$, $n \to \infty$ in probability, uniformly in $t$. Thus, since $S(t-t_0)$ is a bounded operator,
\[
S(t-t_0) u_{n+1}(t) \to S(t-t_0) u^{*}(t), \ n \to \infty.
\]
Next, $\forall \ve > 0$
\[
\P\left\{\left\|\int_{t_0}^{t} S(t-s)[f(u_n(s)) - f(u^{*}(s))] ds\right\|_{H} > \ve \right\}
\]
\[
\leq \P\left\{\int_{t_0}^{t}\left\|S(t-s)[f(u_n(s)) - f(u^{*}(s))]\right\|_{H} ds > \ve \right\}
\]
\[
\leq \P\left\{ L \int_{t_0}^{t} e^{-2(t-s)} \left\|u_n(s) - u^{*}(s)\right\|_{H} ds > \ve \right\} \leq \frac{L}{\ve^2} \sup_{t \in \R} \sqrt{\E (u_n(t) - u^{*}(t))} \to 0, n \to \infty.
\]
So
\[
\int_{t_0}^{t} S(t-s) f(u_n(s))\, ds  \to \int_{t_0}^{t} S(t-s) f(u^*(s)) \, ds
\]
in probability pointwise for every $t \in \R$ as $n \to \infty$. Finally, using Lemma \ref{lem2},
\[
\E \left\|\int_{t_0}^{t} S(t-s)[\sigma(u_n(s)) - \sigma(u^{*}(s))] \, d W(s)\right\|_H^2 \leq \int_{t_0}^{t} \E \left\|S(t-s)[\sigma(u_n(s)) - \sigma(u^{*}(s))]\right\|^2_{H} \, ds
\]
\[
\leq L^2 \sum_{k=1}^{\infty} a_k \int_{t_0}^{t} e^{-4(t-s)} \E \|u_n(s) - u^{*}(s)\|_H^2 \, ds
\leq \frac{L^2}{4} \sum_{k=1}^{\infty} a_k \sup_{t \in \R} \E\|u_n(t) - u^{*}(t)\|_H^2 \to 0, n \to \infty.
\]
It follows from Proposition 4.16 \cite{DapZab92} that
\[
\int_{t_0}^{t} S(t-s)\sigma(u_n(s)) d W(s) \to \int_{t_0}^{t} S(t-s)\sigma(u^*(s)) d W(s)
\]
in probability. Therefore, passing to the limit in (\ref{identity}), we have
\begin{equation}\label{identity1}
u^*(t) = S(t-t_0) u^*(t_0) + \int_{t_0}^{t} S(t-t_0) f(u^*(s)) ds + \int_{t_0}^{t} S(t-s)\sigma(u^*(s)) dW(s)
\end{equation}
The process $u^*$, defined through the integral relation (\ref{identity1}), has continuous trajectories with probability 1. Indeed, while the continuity of the first two terms can be checked straightforwardly, the continuity of the third one is a consequence of the factorization formula \cite[Theorem 5.2.5]{DapZab96}.

We now show that $u^*$ is a stable solution. To this end, let $\eta(t)$ be another solution of (\ref{eqn1semi}) such that $\eta(t_0)$ is $\F_{t_0}$ measurable and $\E \|\eta(t_0)\|^2_H < \infty$.
We have
\begin{eqnarray}
\E \|u^{*}(t) - \eta(t)\|_{H}^2
& \leq & 3 \, \E \left\|S(t-t_0)(u^{*}(t_0) - \eta(t_0)) \right\|_H^2
\nonumber\\
& & + 3\, \E \left( \int_{t_0}^{t} \left\| S(t-s)[f(u^{*}(s)) - f(\eta(s))] \right\|_{H} \, ds \right)^2
\nonumber\\
 & & +  3\, \E
\left\|\int_{t_0}^{t} S(t-s)(\sigma(u^*(s)) - \sigma(\eta(s))) d W(s) \right\|_{H}^{2}.
\label{stab_est}
\end{eqnarray}
Estimating each term separately, we have
\[
\E \left\|S(t-t_0)(u^{*}(t_0) - \eta(t_0)) \right\|_H^2 \leq e^{-4(t-t_0)} \E \|u^{*}(t_0) - \eta(t_0)\|^2_H \leq  e^{-2(t-t_0)} \E \|u^{*}(t_0) - \eta(t_0)\|^2_H,
\]
\[
\E \left( \int_{t_0}^{t} \left\| S(t-s)[f(u^{*}(s)) - f(\eta(s))] \right\|_{H} \, ds \right)^2 \leq \frac{L^2}{2}  \int_{t_0}^{t} e^{-2(t-s)} \E\|u^{*}(s) - \eta(s)\|^2_{H} ds,
\]
and, using Lemma \ref{lem2},
\[
\E \left\|\int_{t_0}^{t} S(t-s)(\sigma(u^*(s)) - \sigma(\eta(s))) d W(s) \right\|_{H}^{2}
\]
\[
\leq L^2  \sum_{k=1}^{\infty} a_k  \int_{t_0}^{t} e^{-4(t-s)} \E\|u^{*}(s) - \eta(s)\|^2_{H} ds
\leq L^2 \sum_{k=1}^{\infty} a_k  \int_{t_0}^{t} e^{-2(t-s)} \E\|u^{*}(s) - \eta(s)\|^2_{H} ds.
\]
Thus (\ref{stab_est}) reads as
\begin{multline}\label{stab_est_imp}
\E \|u^{*}(t) - \eta(t)\|_{H}^2 \leq 3 e^{-2(t-t_0)} \E \|u^{*}(t_0) - \eta(t_0)\|^2_H \\
+ 3 \left(\frac{L^2}{2} + L^2 \sum_{k=1}^{\infty} a_k\right)   \int_{t_0}^{t} e^{-2(t-s)} \E\|u^{*}(s) - \eta(s)\|^2_{H} ds.
\end{multline}
Rewriting (\ref{stab_est_imp}) as
\begin{multline}
e^{2t} \E \|u^{*}(t) - \eta(t)\|_{H}^2 \leq 3 e^{2 t_0} \E \|u^{*}(t_0) - \eta(t_0)\|^2_H + \\
+ 3 L^2 \left(\frac{1}{2} + \sum_{k=1}^{\infty} a_k\right)   \int_{t_0}^{t} e^{2 s} \E\|u^{*}(s) - \eta(s)\|^2_{H} ds,
\end{multline}
we are now in position to apply Gronwall's inequality to conclude that
\[
\E \|u^{*}(t) - \eta(t)\|_{H}^2 \leq 3 e^{\left(-2 +3\left(\frac{L^2}{2} + L^2 \sum_{k=1}^{\infty} a_k\right)\right) (t- t_0)} \E \|u^{*}(t_0) - \eta(t_0)\|^2_H
\]
Thus, $u^{*}$ is stable, provided (\ref{small2}) holds.

The uniqueness of $u^{*}$ can be shown similarly to the linear case.
\end{proof}

\section{Uniqueness of the invariant measure.}\label{SecUniqInvMeas}
In this section we show that the solution $u^{*}(t)$ is a stationary process for $t \in \R$, which defines an invariant measure $\mu$ for
\eqref{eqn1semi}. The stability property of $u^*$ gives the uniqueness of
the invariant measure. We follow the overall procedure in
\cite[Section 11.1]{DapZab92} and \cite[Theorem 6.3.2]{DapZab96}.

Following \cite{DapZab96}, $u^*$ defines a probability transition semigroup
\[
\P_t \f(x):= \E \f(u^*(t,x)), \ x \in H
\]
so that its dual $\P^{*}_t$ is an operator in the space of probability measures $\mu$:
\[
\P^{*}_t \mu(\Gamma) = \int_{H}P_t(u_0,\Gamma) \mu(d u_0), \ t\geq 0, \ \Gamma \subset H.
\]
Here
\[
\P_t(u_0,\Gamma) = \E \chi_{\Gamma}(u(t,u_0)),
\]
and $\chi_{\Gamma}$ is the characteristic function of the set $\Gamma$.
An invariant measure $\mu$ is a fixed point of $\P^{*}_t$, i.e.
$\P^*_t\mu = \mu$ for all $t\geq 0$.

Throughout this section, $u(t,t_0,u_0)$ will denote the solution of
\begin{equation}\label{eqn1t0}
\begin{cases}
\frac{\d}{\d t} u(t,x) = A u(t,x) + f(x,u(t,x)) + \sigma(x,u(t,x)) \dot{W}(t,x),\,\,\, t \geq t_0,\,\,\,x \in G;\\
u(t_0,x) = u_0(x)
\end{cases}
\end{equation}
for $t \geq t_0$. Here, $t_0$ can be any number, in particular
we can choose $t_0=0$.
Also, for any $H$-valued random variable $X$, we use $\mathcal{L}(X)$ to
denote the law of $X$, which is the following measure on $H$:
\[
\mathcal{L}(X)(A):=P(\omega: X(\omega) \in A), A \subset H
\]

We now show that $\mu := \mathcal{L}(u^{*}(t_0))$ is the unique invariant
measure for \eqref{eqn1semi}.
Following \cite[Prop 11.2, 11.4]{DapZab92}, it is sufficient to show that
\begin{equation}\label{weakconv}
\forall u_0 \in H, \,\,\,\,\,\,
P_{t}^{*} \delta_{u_0} = \mathcal{L}(u(t,t_0,u_0)) \to \mu \text{ weakly as } t \to \infty.
\end{equation}

Since the equation is autonomous and $t - t_0 = t_0 - (2t_0 - t))$,
we have the following property of the solution
\begin{equation}\label{distributions}
\mathcal{L}(u(t,t_0,x)) = \mathcal{L}(u(t_0, 2 t_0-t, u_0)),
\,\,\,\,\,\,\text{for all}\,\,\,t > t_0.
\end{equation}
By the stability property
(Definition \ref{defExpStab}) of \eqref{eqn1semi}, we have
\begin{eqnarray*}
\E \|u(t_0, 2 t_0 - t, u_0) - u^{*}(t_0)\|^2_H
& \leq & e^{-2(t_0-2 t_0 + t)} \E \|u(2 t_0 - t, 2 t_0 - t, u_0) - u^{*}(2 t_0 - t)\|^2_H \\
& = & e^{-2(t-t_0)} \E\|u^{*}(2 t_0 - t) - u_0\|_H^2 \to 0, \ t\to \infty,
\end{eqnarray*}
since $\sup_{t \in \R} \E\|u^{*}(2 t_0 - t) - u_0\|_H^2 < \infty$. Thus $u(t_0, 2 t_0 -t, u_0)$ converges in probability to $u^{*}(t_0)$, which, in turn, implies the weak convergence (\ref{weakconv}).
The above simultaneously proves the existence and uniqueness of
the invariant measure for \eqref{eqn1semi}.

The stationarity of $u^*$ follows from \cite[Prop 11.5]{DapZab92}.


\bibliography{bibliographyinv.bib}

\end{document}